\documentclass[10pt]{article}
\usepackage[nointlimits]{amsmath}
\usepackage{amsfonts}\usepackage{amssymb}\usepackage{amsmath}\usepackage{amsthm}
\usepackage{latexsym}
\usepackage{graphicx}
\usepackage{layout}
\usepackage{color}
\newtheorem{theorem}{Theorem}[section]

\newtheorem{example}[theorem]{Example}
\newtheorem{proposition}[theorem]{Proposition}
\newtheorem{lemma}[theorem]{Lemma}
\newtheorem{corollary}[theorem]{Corollary}
\newtheorem{remark}[theorem]{Remark}

\newcommand{\sezione}[1]{\section{#1}\setcounter{equation}{0}}

\def\R{\mathbb R}

\def\e{{\epsilon}}
\def\di12{\mathcal{D}^{1,2}(\R^n)}

\def\d{\delta}

\def\D{\Delta}
\def\l{{\lambda}}

\def\al{\alpha}

\def\th{{\theta}}

\def\0l{_{0,\l}}
\def\1l{_{1,\l}}
\def\2l{_{2,\l}}
\def\3l{_{3,\l}}
\def\4l{_{4,\l}}

\def\de{\partial}

\def\Om{\Omega}

\begin{document}
\title{A Morse Lemma for
degenerate critical points of solutions of nonlinear equations in $\R^2$}
\author{
Massimo Grossi\thanks{Dipartimento di Matematica, Universit\`a di Roma
``La Sapienza", P.le A. Moro 2 - 00185 Roma, e-mail {\sf
massimo.grossi@uniroma1.it}.}}
\maketitle
 \begin{abstract}
In this paper we prove a Morse Lemma for degenerate critical points of a function $u$ which satisfies
\begin{equation}\label{P0}
-\Delta u=f(u) \quad \hbox{ in }B_1,
\end{equation}  
where $u\in C^2(B_1)$, $B_1$ is the unit ball of $\R^2$ and $f$ is a smooth nonlinearity. Other results on the nondegeneracy of the critical points and the shape of the level sets are proved.
 \end{abstract}

\sezione{Introduction and statement of the main results}
A famous result  of  Morse theory concerns the classification of $nondegenerate$ critical points of a smooth function $u$. We recall it in the particular case of the plane,
\begin{lemma}[Morse Lemma]
If $(0,0)$ is a nondegenerate critical point of $u$ then there exists a $C^2$ change of coordinates in a neighborhood of $(0,0)$ such that the
function $u$ expressed with respect to the new local coordinates $(x, y)$ takes
one of the following three standard forms
\begin{itemize}
\item $u(x,y)=u(0,0)-x^2-y^2$\quad if $(0,0)$ is a maximum point for $u$
\item $u(x,y)= u(0,0)+x^2+y^2$\quad if $(0,0)$ is a minimum point for $u$
\item $u(x,y)=u(0,0)+x^2-y^2$\quad if $(0,0)$ is a saddle point for $u$.
\end{itemize}
\end{lemma}
By  saddle point we mean a critical point for $u$ which is neither a maximum nor a minimum. We recall that $(0,0)$ is a  $nondegenerate$ critical point of a function $u\in C^2(B_1)$ is if the Hessian matrix $H_u\big((0,0)\big)=\left(\begin{matrix}u_{xx}(0,0)&u_{xy}(0,0)\\u_{xy}(0,0)&u_{yy}(0,0)\end{matrix}\right)$ is invertible.

A lot of study has been devoted to the case where the critical point is degenerate mainly addressed to study the topological properties of the  sublevels of $u$ (critical groups, Betti number, etc.). It is virtually impossible to provide a complete bibliography on this subject, we simply mention the classic Conley and Milnor's books  (\cite{C} and \cite{M}) and the paper \cite{D} where some applications to \eqref{P} have been considered.

Anyway if the critical point is $degenerate$, a complete classification like in the Morse Lemma is impossible without some additional assumptions. Here we require that $u$ is a solution to
\begin{equation}\label{P}
-\Delta u=f(u) \quad \hbox{ in }B_1,
\end{equation}
where $u\in C^2(B_1)$ and $B_1$ is the unit ball of $\R^2$. We assume that $f:\R\to\R$ is a smooth nonlinearity, say $f\in C^\infty$. In many cases this assumption can be relaxed, we do not care about the optimal regularity  for $f$.

In this paper we want to extend the Morse Lemma to $degenerate$ critical points of solutions of the problem  \eqref{P}. We would like to stress that
\begin{itemize}
\item We do not assume any sign assumption on $f$,
\item Our result is local, so no boundary condition is needed,
\item  The critical point $(0,0)$ does not need to be isolated.
\end{itemize}

In order to state our results we need to fix a suitable setting. For a complex number $z=x+iy\in\mathbb{C}$ let us denote by $Re(z)$ and $Im(z)$ the real and imaginary part of $z$. Moreover if $(0,0)$ is a  degenerate critical point of $u$ it is not restrictive to assume that (up to a suitable rotation)
\begin{equation}\label{c1}
u_{xx}(0,0)=u_{xy}(0,0)=0.
\end{equation}
Finally  if $u_x\not\equiv0$ let us consider the minimum integer $n\ge3$ such that
\begin{equation}\label{c2a}
\frac{\partial^nu}{\partial x^{n-k}\partial y^k}(0,0)\ne0\quad\hbox{for some }k=0,..,n-1.
\end{equation}
By classical results (see \cite{Ar} or page 422 in \cite{CF}) if $u_x\not\equiv0$ such a $n$ always exists.
In all the paper we refer to $n$ as the integer satisfying \eqref{c2a}. 

Our first result consider the case where $(0,0)$ is a  degenerate maximum point. 
\begin{theorem}[Morse Lemma for degenerate maximum points]\label{c2}
Let $u$ be a nonconstant solution to  \eqref{P}. Then if $(0,0)$ is a degenerate maximum point of $u$ satisfying \eqref{c1} then
\begin{equation}\label{c1a}
u_{yy}(0,0)<0.
\end{equation}
Moreover the following expansions hold:\\
$i)$ If $n$ is even we have that
\begin{equation}\label{c1b}
u(x,y)=u(0,0)+\left(\frac{u_{yy}(0,0)}2+o(1)\right)y^2+\left(\frac{\frac{\partial^nu}{\partial x^n}(0,0)}{n!}+o(1)\right)Re\left(z^n\right)+\left(
\frac{\frac{\partial^nu}{\partial x^{n-1}\partial y}}{n!}(0,0)+o(1)\right)Im\left(z^n\right),
\end{equation}
with $\frac{\partial^nu}{\partial x^n}(0,0)<0$.\\
$ii)$ If $n$ is odd we have that $\frac{\partial^nu}{\partial x^n}(0,0)=0$, $\frac{\partial^nu}{\partial x^{n-1}\partial y}(0,0)\ne0$ and
 there exists an integer $l\in[n+1,2n-2]$ such that
\begin{equation}\label{g1}
\frac{\partial^lu}{\partial x^l}(0,0)<0
\end{equation}
and 
\begin{equation}\label{c1c}
u(x,y)=u(0,0)+\left(\frac{u_{yy}(0,0)}2+o(1)\right)y^2+\left(
\frac{\frac{\partial^nu}{\partial x^{n-1}\partial y}(0,0)}{n!}+o(1)\right)Im\left(z^n\right)+\left(\frac{\frac{\partial^lu}{\partial x^l}(0,0)}{l!}+o(1)\right)Re\left(z^l\right).
\end{equation}
Finally if $l=2n-2$ we have that
\begin{equation}\label{c1d}
\left(\frac{\partial^nu}{\partial x^{n-1}\partial y}(0,0)\right)^2\le \frac{2\big[(n-1)!\big]^2}{(2n-2)!}u_{yy}(0,0) \frac{\partial^{2n-2}u}{\partial x^{2n-2}}(0,0).
\end{equation}
\end{theorem}
Before to give an idea of the proof let us make some comments on Theorem \ref{c2}.
\begin{remark}
The same result holds if $(0,0)$ is a degenerate minimum point of $u$. In this case  $u_{yy}(0,0)>0$, $\frac{\partial^nu}{\partial x^n}(0,0)>0$, $\frac{\partial^lu}{\partial x^l}(0,0)>0$ and \eqref{c1d} does not change.
\end{remark}
\begin{remark}
Note that in Theorem  \ref{c2} the assumption $u_x\not\equiv0$ is not restrictive. Indeed if  $u_x\equiv0$ then \eqref{P} becomes 
\begin{equation}
\begin{cases}
-u''=f(u)\\
u'(0)=0
\end{cases}
\end{equation}
and expansions of the solution $u$ follows immediately. This case appears 
when $u(x,y)=\cos y$ which verifies $-\Delta u =u$ and admits $y=0$ as maximum points.
\end{remark}
In our opinion \eqref{c1a} is the most relevant results of the Theorems (see \cite{AP} for some properties of solutions satisfying  \eqref{c4b}). In  particular we get that there is no $u$ verifying \eqref{P} such that $u(x,y)\sim u(0,0)-x^4-y^4$ in a neighborhood of $(0,0)$!

Some similar properties to \eqref{c1b} and \eqref{c1c} can be found in \cite{CF} where the authors study the properties of the zero-set of solutions $u$ to
\begin{equation}
\Delta u=f(x,u,\nabla u)
\end{equation}
with
\begin{equation}
|f(x,u,\nabla u)|\le A|u|^\al+B|u|^\beta,\quad A,B>0,\ \al,\beta\ge1.
\end{equation}
Our result can be seen as an extension of Theorem 1.2 in \cite{CF} in an appropriate setting.

An interesting particular case of Theorem \ref{c2} is when the maximum (minimum) point is non-isolated.
Here we show that  equality holds in  \eqref{c1d}.
\begin{proposition}\label{P1}
Let $u$ be a nonconstant solution to  \eqref{P} and $(0,0)$ a  not $isolated$ degenerate maximum point of $u$ satisfying \eqref{c1}.

Then  $n$ is odd, $l=2n-2$ and the equality holds in  \eqref{c1d}.
\end{proposition}
A very studied case is that of $radial$ functions. Here we have that always occurs $n=3$ in Theorem \ref{c2}.
\begin{corollary}\label{C3}
Let $u$ be a nonconstant solution to  \eqref{P} and suppose that $(0,0)$ verifies \eqref{c1} and it is a maximum point
of a radial function $u=u(r)$ with $r^2=x^2+(y-P)^2$, $P\ne0$. Then the following expansion holds
\begin{equation}\label{e6}
\begin{split}
&u(x,y)=u(0,0)-\left(\frac{f(u(P))}2+o(1)\right)y^2+\left(\frac{f(u(P))}{6P}+o(1)\right)\left(3x^2y-y^3\right)-\\
&\left(\frac{f(u(P))}{8P^2}+o(1)\right)\left(x^4-6x^2y^2+y^4\right),
\end{split}
\end{equation}
and
\begin{equation}\label{e7}
\left(u_{xxy}(0,0)\right)^2=\frac{f^2(u(P))}{P^2}=\frac{u_{yy}(0,0)u_{xxxx}(0,0)}3.
\end{equation}
\end{corollary}
Next step is to get an analogous  of Theorem \ref{c2} when $(0,0)$ is a saddle point of a solution $u$ to \eqref{P}. Here  we cannot expect that \eqref{c1a} holds for $any$ saddle point to $u$.  Indeed if we consider
\begin{equation}
u(x,y)=Re(z^n)\quad\hbox{with }n\ge3,
\end{equation}
we get that $u$ satisfies $\Delta u=0$, $u_{yy}(0,0)=0$ and $(0,0)$ is a degenerate saddle point. Note that in this case $i[u,(0,0)]=1-n\le-2$ where $i[u,(0,0)]$ denotes the $index$ of $\nabla u$ at $(0,0)$.\\
Recall that if $(0,0)$ is an isolated critical point then the $index$ of $\nabla u$ at  $(0,0)$ is given by (denoting by $B(0,\e)$ the ball centered at the origin and radius $\e$),
\begin{equation}
i[\nabla u, (0,0)]=\lim\limits_{\e\to0}\deg\big(\nabla u, B(0,\e),(0,0)\big)
\end{equation}
Next result shows that  if $(0,0)$ is an $isolated$ saddle point with index greater than $-2$ then  the condition $u_{yy}(0,0)\ne0$ is again verified.
\begin{theorem}[Morse Lemma for degenerate saddle points]\label{c3b}
Let $u$ be a solution to  \eqref{P} and assume  \eqref{c1}. Then if $(0,0)$ is an  isolated degenerate saddle point of $u$  verifying
\begin{equation}
i[u,(0,0)]\ge-1.
\end{equation}
then we have that 
\begin{equation}\label{c4b}
u_{yy}(0,0)\ne0.
\end{equation}
Moreover if  $i[u,(0,0)]=-1$ the following expansions hold:\\
$i)$ If $n$ is even we have that
\begin{equation}\label{c4c}
u(x,y)=u(0,0)+\frac{u_{yy}(0,0)+o(1)}2y^2+\frac{\frac{\partial^nu}{\partial x^n}(0,0)+o(1)}{n!}Re\left(z^n\right)+
\frac{\frac{\partial^nu}{\partial x^{n-1}y}(0,0)+o(1)}{n!}Im\left(z^n\right),
\end{equation}
and if  $\frac{\partial^nu}{\partial x^n}(0,0)\ne0$ then 
\begin{equation}\label{c4a}
\frac{\partial^nu}{\partial x^n}(0,0)u_{yy}(0,0)<0.
\end{equation}
$ii)$ If $n$ is odd we have that 
\begin{equation}\label{e1c}
u(x,y)=u(0,0)+\frac{u_{yy}(0,0)+o(1)}2y^2+
\frac{\left(\frac{\partial^nu}{\partial x^{n-1}y}(0,0)+o(1)\right)}{n!}Im\left(z^n\right).
\end{equation}
Finally if $i[u,(0,0)]=0$ we can only say that
\begin{equation}
u(x,y)=u(0,0)+\frac{u_{yy}(0,0)+o(1)}2y^2+\frac{\frac{\partial^nu}{\partial x^n}(0,0)+o(1)}{n!}Re\left(z^n\right)+
\frac{\frac{\partial^nu}{\partial x^{n-1}y}(0,0)+o(1)}{n!}Im\left(z^n\right).
\end{equation}
\end{theorem}
 We stress that \eqref{c4b} holds even if $i[u,(0,0)]=0$. In general this is the ``worst'' critical point to handle since its existence does not imply a change of topology in the sub-levels of $u$. A typical example is given by $u(x,y)=y^2-x^3+3xy^2$. Despite this \eqref{c4b} is still valid.

Now we describe the proof of Theorems  \ref{c2} and \ref{c3b}. The basic idea is to mix  some algebraic identities satisfied by the derivatives of $u$ and topological properties of $\nabla u$.

The first step (Proposition \ref{c17b}) is to prove that if $u_{yy}(0,0)=0$ then necessarily $(0,0)$ is $isolated$ and furthermore we have that
\begin{equation}
i[\nabla u, (0,0)]\le-2
\end{equation}
where $i[\nabla u, (0,0)]$ denotes the index of $\nabla u$ at $(0,0)$. 
In this way, recalling that a maximum (minimum) point has index $1$ (see Theorem \ref{b3}),
we have the claim of \eqref{c1a} and \eqref{c4b}.
Next, assuming that \eqref{c4b} holds, we use again some algebraic identities satisfied by the derivatives of $u$ in order to write the following expansion, (Proposition \ref{c17b})
\begin{equation}\label{l1}
\begin{split}
&u(x,y)=u(0,0)+\frac{u_{yy}(0,0)+o(1)}2y^2+\frac{\frac{\partial^nu}{\partial x^n}(0,0)}{n!}Re\left(z^n\right)+
\frac{\frac{\partial^nu}{\partial x^{n-1}y}(0,0)}{n!}Im\left(z^n\right)+R(x,y).
\end{split}
\end{equation}
with $R(x,y)=O\left(|x|^n+|y|^n\right)$ and $n$ is given by \eqref{c2a}. If $(0,0)$ is a a maximum point from \eqref{l1} we deduce \eqref{c1b}- \eqref{c1d}. If $u$ is a saddle point, of course is more difficult to deduce general properties of the coefficients of \eqref{c2}. However, if $i[\nabla u, (0,0)]=-1$, some topological arguments allows to deduce \eqref{c4a} and \eqref{e1c}.
We believe that Theorems \ref{c2} and \ref{c3b} are a very useful tool to study qualitative properties of solutions to \eqref{P}. 

The last part of the paper (Section \ref{s4}) is strongly influenced by the paper \cite{CC}. Here
we assume that $u$ has a $unique$ critical point satisfying the problem,
\begin{equation}\label{P+}
\begin{cases}
-\Delta u=f(u) & \hbox{ in }\Omega\\
u>0  & \hbox{ in }\Omega\\
u=0 & \hbox{ on }\de \Omega.
\end{cases}
\end{equation}
and extend  some results of  \cite{CC} without to require that $u$ is a semi-stable solution to \eqref{P}.\\
More precisely we want to understand if the maximum point it degenerate or not jointly with 
the properties of the level set of $u$. First let us recall some known results where nondegeneracy is proved,
\begin{itemize}
\item The Gidas,Ni, Nirenberg Theorem (\cite{GNN}) in convex and symmetric domains.
\item The Cabr\'e-Chanillo  Theorem (\cite{CC}) for semi-stable solutions in strictly convex domains.
\item Solutions which concentrate at some point (\cite{GG}).
\end{itemize}
All the previous example will be discussed with more details in Section \ref{s2}.\\
Now we want to explore more closely the properties of the solution $u$ according the degeneracy of its critical points. Our first result generalizes those considered in the previous examples because neither the symmetry of the domain nor any hypothesis on the solution is required.
\begin{theorem}\label{T1}
Let us consider a solution $u$ to \eqref{P+} and assume that 
\begin{equation}
\partial\Om\hbox{ has strictly positive curvature },
\end{equation}
\begin{equation}\label{d6a}
f(0)\ge0,
\end{equation}.
\begin{equation}\label{d3}
(0,0) \hbox{ is the only critical point of $u$ in }\Om.
\end{equation}
Then $(0,0)$ is nondegenerate.
\end{theorem}
Note that \eqref{d6a} is used to apply the Hopf Lemma to the boundary of $\Om$ but it
 can be relaxed to $\nabla u\ne(0,0)$ on $\partial\Om$ (see Remark \ref{d3a}). 
 
The proof of Theorem \ref{T1} relies on the study of the zero-set  of some partial derivatives of $u$. Similar techniques was used in \cite{DGP} (see also \cite{P}).
An interesting consequence of the previous theorem is the following result,
\begin{corollary}\label{C1}
Let us consider a solution $u$ to \eqref{P+} with $f(0)\ge0$. Suppose that
$(0,0)$ is the only critical point of $u$. Then the following alternative holds:
either
\begin{equation}
(0,0) \hbox{ is a nondegenerate critical point for }u
\end{equation}
or
\begin{equation}\label{d5}
 \hbox{all level sets of $u$ have a points with nonpositive curvature}.
\end{equation}
\end{corollary}
A curious consequence of the previous proposition is  that the behavior of the level sets of a solution to \eqref{P+} seems to be more difficult to predict if the maximum point of $u$ is non-degenerate! In fact in this case we can have both convex and non-convex super-level sets (the latter case appears in \cite{HNS}). An explicit example where \eqref{d5} holds is given in Example \ref{e5}.

Another consequence of Theorem \ref{T1} concerns solutions in the whole space,
\begin{corollary}\label{C2}
Suppose that $u$ verifies
\begin{equation}
-\Delta u=f(u)\quad\hbox{ in }\R^2
\end{equation}
with $f(0)\ge0$. Assume that $(0,0)$ is the only critical point of $u$ and all level set of $u$ are Jordan curves.\\
Then the following alternative holds:
either
\begin{equation}
(0,0) \hbox{ is a nondegenerate critical point for }u
\end{equation}
or
\begin{equation}
 \hbox{all level sets of $u$ have a points with nonpositive curvature}.
\end{equation}
\end{corollary}
We end this introduction pointing out that Theorems \ref{c2} and \ref{c3b} suggest some nice explicit examples. In one of them we prove the existence of a star-shaped domain $\Om$ where a semi-stable solution for \eqref{P+} admits exactly $two$ critical points. This proves that the assumption on the positivity of the  curvature of $\partial\Om$ in Cabr\'e-Chanillo's result cannot be relaxed.

Our project is continue to study properties of the level sets of $u$ and the degeneracy of its critical points when $u$ admits two or more critical points. 

The paper is organized as follows: in Section \ref{s2} we recall some useful preliminaries. In Section \ref{s3} we prove Theorems \ref{c2} and \ref{c3b}. In Section \ref{s4} we prove Theorem \ref{T1} and its consequences. Finally in Section \ref{s5} we prove Proposition \ref{P1} and Corollary \ref{C3}.\vskip0.2 cm\noindent
{\bf Acknowledgement}. 
I would like to thank Francesca Gladiali. for her peerless help without which this paper would have
never been written. Grazie Francesca!

\sezione{Known results}\label{s2}
The first result of this section is a direct consequence of the Gidas, Ni, Nirenberg Theorem (\cite{GNN}). Since in \cite{GNN} it  is not explicitly stated we give the proof.
\begin{theorem}\label{f1}
Let $\Om$ an arbitrary bounded domain in $\R^N$ which is convex in the $x_i$ direction and symmetric with respect to the plane $x_i=0$ for any $i=1,..,N$. Let $u\in C^2(\Om)\cap C(\overline\Om)$ be a positive solution to
\begin{equation}
\begin{cases}
-\Delta u=f(u) & \hbox{ in }\Omega\\
u>0  & \hbox{ in }\Omega\\
u=0 & \hbox{ on }\de \Omega.
\end{cases}
\end{equation}
with $f\in C^1(\Om)$. Then the Hessian matrix at  the origin is diagonal and strictly negative definite.
\end{theorem}
\begin{proof}
We give the proof only for $N=2$, there is no difference in higher dimensions. By the Gidas, Ni, Nirenberg Theorem the function $v=\frac{\partial u}{\partial x_1}$ satisfies
\begin{equation}
\begin{cases}
-\Delta v=f'(u)v & \hbox{ in }\Omega\cap\{x_1<0\}\\
v>0  & \hbox{ in }\Omega\cap\{x_1<0\}\\
v=0 & \hbox{ on }\Om\cap\{x_1=0\}.
\end{cases}
\end{equation}
Since $v(0,x_2)=0$ in $\Om$ we get that $0=\frac{\partial v}{\partial x_2}(0,0)=\frac{\partial^2u}{\partial x_1\partial x_2}(0,0)$ and by Hopf Lemma it follows that $\frac{\partial^2u}{\partial x_1^2}(0,0)<0$. In the same way we get that $\frac{\partial^2u}{\partial x_2^2}(0,0)<0$ which ends the proof.
\end{proof}
The second result, due to Cabr\'e and Chanillo,  removes the symmetry assumption requiring that $u$ belongs to a suitable class of solutions.
\begin{theorem}\label{a1a}
Let $\Om$ be a smooth bounded and convex domain in $\R^2$ whose boundary has positive curvature. Let $u$ be a smooth nontrivial solution to
\begin{equation}\label{a1}
\begin{cases}
-\Delta u=f(u) & \hbox{ in }\Omega\\
u=0 & \hbox{ on }\de \Omega.
\end{cases}
\end{equation}
with $f\in C^1(\Om)$ and $f\ge0$. Moreover assume that $u$ is a semi-stable solution to \eqref{a1}, i.e. the first eigenvalue of the  linearized operator $\mathcal{L}$ to \eqref{a1} at $u$ defined as
$$\mathcal{L}=-\D-f'(u)I$$
is nonnegative. Then $u$ has a unique critical point $x_0$ in $\Om$. Moreover $x_0$ is the maximum of $u$ and it is nondegenerate.
\end{theorem}
\begin{proof}
See \cite{CC}.
\end{proof}
The last result concerns solutions which concentrate at some points. There are several results of this type (also in higher dimensions) for various nonlinearities. We just mention one of them.
\begin{theorem}\label{a1b}
Let $\Om\subset\R^2$ be a smooth bounded domain and $u_\l$ a solution to 
\begin{equation}
\begin{cases}
-\Delta u=\lambda e^u & \hbox{ in }\Omega\\
u=0 & \hbox{ on }\de \Omega.
\end{cases}
\end{equation}
satisfying $\lambda\int_\Om e^u\to8\pi$ as $\l\to0$. Let $x_\l\in\Om$ the point where $u_\l(x_\l)=||u_\l||_\infty$. Then $x_\l$ is a nondegenerate critical point of $u_\l$ for $\l$ small enough.
\end{theorem}
\begin{proof}
In Lemma $5$ of \cite{GG} it was shown that $z_\l(x)=u_\l(\d_\l x+x_\l)-u_\l(x_\l)\to\log\frac1{\left(1+\frac{|x|^2}8\right)^2}$ in $C^2_{loc}(\R^2)$ with $\d_\l^2=\frac1{\l e^{u_\l(x_\l)}}\to0$. Hence  for $i,j=1,2$ we get 
$$\det\left\{\frac{\partial^2u}{\partial x_i\partial x_j}(x_\l)\right\}=\det\left\{\frac{\partial^2z}{\partial x_i\partial x_j}(0,0)\right\}\to\frac14,$$
and so the claim follows.
\end{proof}
Next result is classical and it will be used in the proof of Theorem \ref{c2}.
\begin{theorem}\label{b3}
Let $(0,0)$ be an isolated minimum (maximum) point of $f:B_1\to\R$ where $B_1\subset\R^2$. Then
\begin{equation}
i[\nabla f,(0,0)]=1
\end{equation}
where $i(\nabla f,(0,0))$ denotes the index of $f$ at $(0,0)$.
\end{theorem}
\begin{proof}
See \cite{A}.
\end{proof}
In the proof of Theorem \ref{c3b} we need to compute the index of some suitable vector field. Next
 lemma  will be useful  for it.
\begin{lemma}\label{c12}
$m>1$ being an integer, and $B_1\subset\R^2$ being an open neighborhood of $(0,0)$, assume that $f\in C(B_1,\R^2)$ can be written in the form
\begin{equation}
f(x,y)=L(x,y)+H_m(x,y)+R(x,y)
\end{equation}
where $L:\R^2\to\R^2$ is linear and not invertible, $H_m$ is homogeneous of order $m$ and $R(x,y)$ verifies
\begin{equation}
\lim\limits_{(x,y)\to(0.0)}\frac{R(x,y)}{(x^2+y^2)^\frac m2}=0.
\end{equation}
If $H_m(x,y)\not\in Im(L)$ for any $x(x,y)\in Ker(L)\setminus\{(0,0)\}$, $i(f,(0,0))$ is well defined and
\begin{equation}
i(f,(0,0))=i(L+JP,(0,0))\cdot i\left(J^{-1}QH_m\big|_{Ker(L)},(0,0)\right)
\end{equation}
where $Q$ is any projector such that $Ker(Q)=Im(L)$ and $J:Ker(L)=Im(Q)$ any isomorphism.
\end{lemma}
\begin{proof}
See Corollary 6.5.1, page 51 in  \cite{DM}.
\end{proof}

\sezione{Proof of Theorems \ref{c2} and \ref{c3b}}\label{s3}
In this section we prove Theorems \ref{c2} and \ref{c3b}. We have the following 
\begin{proposition}\label{e3b}
Let us suppose that $(0,0)$ is a degenerate critical point of a nonconstant function $u$ satisfying \eqref{P} and  \eqref{c1}. Then if
\begin{equation}
u_{yy}(0,0)=0
\end{equation}
we have that 
 $i[\nabla u,(0,0)]$ is well defined and
\begin{equation}\label{c3}
i[\nabla u,(0,0)]\le-2.
\end{equation}
\end{proposition}
\begin{proof}
Differentiating \eqref{P} with respect to $x$ and $y$ we get that $u_x$ and $u_y$ are  solutions to
\begin{equation}\label{c4}
-\Delta v=f'(u)v \quad \hbox{ in }\Omega.
\end{equation}
First let us consider the case where $u_x\equiv0$ and consider $z(y)=u_y(0,y)$ which satisfies (since $u_{yy}(0,0)=0$)
\begin{equation}
\begin{cases}
-z''=f'(u(0,y))z&\text{ in }(-\delta,\delta)\\
z(0)=z'(0)=0
\end{cases}
\end{equation}
So $z\equiv0$ and $u(0,y)$ is constant. Since $u_x\equiv0$ we get that $u$ is constant in $B_1$ which is a contradiction with our assumption.\\
Hence $u_x\not\equiv0$ and since $u_x$ solves \eqref{c4} as pointed out in the introduction there is an integer $n\ge3$ such that 
\begin{equation}\label{e4b}
\frac{\partial^nu}{\partial y^{n-k}\partial x^k}(0,0)\ne0\quad\hbox{for some }1\le k\le n.
\end{equation}
Let us choose the minimum integer $n$ such that \eqref{e4b} holds for some $1\le k\le n$.\\
Next we prove that all derivatives of $u$ with order less than $n$ are zero, i.e.
\begin{equation}\label{C4a}
\frac{\partial^mu}{\partial y^{m-h}\partial x^h}(0,0)=0\quad\hbox{for any }0\le h\le m<n,\ m\ge3.
\end{equation}
This is obvious if $h>0$ by the definition of $n$. Hence consider $h=0$ and by contradiction suppose  that 
\begin{equation}\label{C4b}
\frac{\partial^mu}{\partial y^m}(0,0)\ne0
\end{equation}
where we take the minimum integer $m$ such that \eqref{C4b} holds. Differentiating \eqref{P} $m-2$ times we get (using the minimality of $m$)
\begin{equation}
0\ne\frac{\partial^mu}{\partial y^m}(0,0)=-\frac{\partial^mu}{\partial y^{m-2}\partial x^2}(0,0)
\end{equation}
which is not possible again by the definition of $n$. Hence \eqref{C4a} holds.

Finally differentiate  \eqref{P} $n-2$ times with respect to $x$ and $y$. We get that
\begin{equation}\label{c5}
\frac{\partial^nu}{\partial y^{n-k}\partial x^k}(0,0)=-\frac{\partial^nu}{\partial y^{n+2-k}\partial x^{k-2}}(0,0)\qquad\hbox{ for any } k=2,..,n,\ n>2
\end{equation}
An important consequence of \eqref{c5} is that 
\begin{equation}\label{c5a}
\begin{cases}
\frac{\partial^nu}{\partial x^n}(0,0)=-\frac{\partial^nu}{\partial x^{n-2}\partial y^2}(0,0)=\dots\Rightarrow
\frac{\partial^nu}{\partial x^{n-2h}\partial y^{2h}}(0,0)=(-1)^h\frac{\partial^nu}{\partial x^n}(0,0),\\ \\
\frac{\partial^nu}{\partial x^{n-1}\partial y}(0,0)=-\frac{\partial^nu}{\partial x^{n-3}\partial y^3}(0,0)=\dots\Rightarrow
\frac{\partial^nu}{\partial x^{n-2h-1}\partial y^{2h+1}}(0,0)=(-1)^h\frac{\partial^nu}{\partial x^{n-1}\partial y}(0,0).
\end{cases}
\end{equation}
This gives, by the definition of $n$, that
\begin{equation}\label{e4a}
\left(\frac{\partial^nu}{\partial x^n}(0,0),\frac{\partial^nu}{\partial x^{n-1}\partial y}(0,0)\right)\ne(0,0)
\end{equation}
The previous computations allow to write the following  Taylor formula for $u$,
\begin{equation}\label{C5a}
\begin{split}
&n!u(x,y)=n!u(0,0)+\sum\limits_{k=0}^n\left(\begin{matrix}n\\k\end{matrix}\right)
\frac{\partial^nu}{\partial x^{n-k}\partial y^k}(0,0)x^{n-k}y^k+R(x,y)=n!u(0,0)+\\
&\sum\limits_{h=0}^{\left[\frac n2\right]}\left(\begin{matrix}n\\2h\end{matrix}\right)
\frac{\partial^nu}{\partial x^{n-2h}\partial y^{2h}}(0,0)x^{n-2h}y^{2h}+
\sum\limits_{h=0}^{\left[\frac{n-1}2\right]}\left(\begin{matrix}n\\2h+1\end{matrix}\right)
\frac{\partial^nu}{\partial x^{n-2h-1}\partial y^{2h+1}}(0,0)x^{n-2h-1}y^{2h+1}+\\
&+R(x,y)=\quad(\hbox{using }\eqref{c5a})=n!u(0,0)+\\
&\frac{\partial^nu}{\partial x^n}(0,0)\sum\limits_{h=0}^{\left[\frac n2\right]}\left(\begin{matrix}n\\2h\end{matrix}\right)
(-1)^kx^{n-2h}y^{2h}+\frac{\partial^nu}{\partial x^{n-1}\partial y}(0,0)
\sum\limits_{h=0}^{\left[\frac{n-1}2\right]}\left(\begin{matrix}n\\2h+1\end{matrix}\right)
(-1)^hx^{n-2h-1}y^{2h+1}+R(x,y)=\\
&n!u(0,0)+\frac{\partial^nu}{\partial x^n}(0,0)Re\left(z^n\right)+\frac{\partial^nu}{\partial x^{n-1}\partial y}(0,0)Im\left(z^n\right)+R(x,y),
\end{split}
\end{equation}
with $R(x,y)=O\left(|x|^{n+1}+|y|^{n+1}\right)$. 
Note that in the last line of \eqref{C5a} we used that
\begin{equation}
\begin{split}
&z^n=(x+iy)^n=\underbrace{\sum\limits_{k=0}^{\left[\frac n2\right]}\left(\begin{matrix}n\\2k\end{matrix}\right)(-1)^kx^{n-2k}y^{2k}}_{Re\left(z^n\right)}+i\underbrace{\sum\limits_{k=0}
^{\left[\frac{n-1}2\right]}\left(\begin{matrix}n\\2k+1\end{matrix}\right)(-1)^kx^{n-2k-1}y^{2k+1}}_{Im\left(z^n\right)}.
\end{split}
\end{equation}
Then differentiating \eqref{C5a} with respect to $x$ and $y$ we get
\begin{equation}\label{c9}
\begin{cases}
u_x(x,y)=&\frac{\partial^nu}{\partial x^n}(0,0)\frac1{(n-1)!}Re\left(z^{n-1}\right)+\frac{\partial^nu}{\partial x^{n-1}\partial y}(0,0)\frac1{(n-1)!}Im\left(z^{n-1}\right)+
\frac{\partial R(x,y)}{\partial x}\\ \\
u_y(x,y)=&-\frac{\partial^nu}{\partial x^n}(0,0)\frac1{(n-1)!}Im\left(z^{n-1}\right)+\frac{\partial^nu}{\partial x^{n-1}\partial y}(0,0)\frac1{(n-1)!}Re\left(z^{n-1}\right)+\frac{\partial R(x,y)}{\partial y}
\end{cases}
\end{equation}
with
\begin{equation}
\left|\frac{\partial R(x,y)}{\partial x}\right|,\left|\frac{\partial R(x,y)}{\partial y}\right|\le C\sum_{k=0}^n|x|^{n-k}|y|^k \le C\left(|x|^n+|y|^n\right)\le C|z|^n
\end{equation}
 Set
\begin{equation}
\al=\frac{\partial^nu}{\partial x^n}(0,0),\quad\beta=\frac{\partial^nu}{\partial x^{n-1}\partial y}(0,0).
\end{equation}
A first consequence of \eqref{c9} is that $(0,0)$ is an $isolated$ critical point. Indeed if $(u_x(x,y),u_y(x,y))=(0,0)$  by  \eqref{c9} we get
\begin{equation}\label{c10}
\begin{split}
&(\alpha^2+\beta^2)\left(\left[Re\left(z^{n-1}\right)\right]^2+\left[Im\left(z^{n-1}\right)\right]^2\right)=\big[(n-1)!\big]^2
\left[\left(\frac{\partial R(x,y)}{\partial x}\right)^2+\left(\frac{\partial R(x,y)}{\partial y}\right)^2\right]\\
&\Leftrightarrow(\alpha^2+\beta^2)|z|^{2n-2}\le C|z|^{2n}
\end{split}
\end{equation}
and since $(\alpha,\beta)\ne(0,0)$ we get that any other critical point satisfies $|z|\ge C$. Hence the index $i(\nabla u,(0,0))$ is well  defined. \\
The final step of the proof is to show that $i(\nabla u,(0,0))\le-2$. Let us introduce the vector field
\begin{equation}
\begin{split}
&\mathcal{L}=(\mathcal{L}_1,\mathcal{L}_2)=\\
&\left(\frac \alpha{(n-1)!}Re\left(z^{n-1}\right)+\frac \beta{(n-1)!}Im\left(z^{n-1}\right),\frac\beta{(n-1)!}Re\left(z^{n-1}\right)-\frac\alpha{(n-1)!}Im\left(z^{n-1}\right)
\right)
\end{split}
\end{equation}
and the following homotopy for $ t\in[0,1]$,
\begin{equation}
\begin{split}
&H(t,x,y)=t\nabla u+(1-t)\mathcal{L}.
\end{split}
\end{equation}
Let us show that $H(t,x,y)\ne(0,0)$ for any $t\in[0,1]$ and $x^2+y^2=\delta>0$ small. By contradiction suppose that $H(t,x,y)=(0,0)$ for some $t\in[0,1]$ and $x^2+y^2=\delta$ small enough. Then arguing as in \eqref{c10} we get a contradiction. This means that
\begin{equation}
i(\nabla u,(0,0))=i(\mathcal{L},(0,0)),
\end{equation}
(in other words we can neglect the remainder term $\nabla R$ in \eqref{c9}). Finally let us compute $i(\mathcal{L},(0,0))$. Suppose that $\al\ne0$ (if $\al=0$ then by \eqref{e4a} $\beta\ne0$ and the proof is the same) and set
\begin{equation}
\begin{split}
&\mathcal{M}= \alpha{(n-1)!}\Big(Re\left(z^{n-1}\right),-Im\left(z^{n-1}\Big)
\right).
\end{split}
\end{equation}
Using again the homotopy $H(t,x,y)=t\mathcal{L}+(1-t)\mathcal{M}$ for $t\in[0,1]$ we get that  $H(t,x,y)=(0,0)$ implies
\begin{equation}
\begin{cases}
\alpha Re\left(z^{n-1}\right)+t\beta Im\left(z^{n-1}\right)=0\\
t\beta Re\left(z^{n-1}\right)-t\alpha Im\left(z^{n-1}\right)=0
\end{cases}
\end{equation}
and then $z=0$. This implies that
\begin{equation}
i(\mathcal{L},(0,0))=i(\mathcal{M},(0,0)),
\end{equation}
and by known arguments  (see for example Theorem 3.1 in \cite{DKV})
\begin{equation}
i(\mathcal{M},(0,0))=1-n.
\end{equation}
So we have that
\begin{equation}
i(\nabla u,(0,0))=1-n\le -2
\end{equation}
since $n\ge3$. Then \eqref{c3} follows.
\end{proof}
\begin{proposition}\label{c17b}
Let us suppose that $(0,0)$ is a degenerate critical point of $u$ satisfying \eqref{P} and  \eqref{c1} with
\begin{equation}\label{c14a}
i[\nabla u,(0,0)]>-2.
\end{equation} 
Then if $u_x\not\equiv0$  the following expansion holds,
\begin{equation}\label{e18}
\begin{split}
&u(x,y)=u(0,0)+\frac{u_{yy}(0,0)+o(1)}2y^2+\frac{\frac{\partial^nu}{\partial x^n}(0,0)}{n!}Re\left(z^n\right)+
\frac{\frac{\partial^nu}{\partial x^{n-1}\partial y}(0,0)}{n!}Im\left(z^n\right)+R(x,y).
\end{split}
\end{equation}
with $R(x,y)=O\left(|x|^n+|y|^n\right)$ and $n$ is given by \eqref{c2a}. Moreover
\begin{equation}\label{C13}
\left(\frac{\partial^nu}{\partial x^n}(0,0),\frac{\partial^nu}{\partial x^{n-1}\partial y}(0,0)\right)\ne(0,0)
\end{equation}
\end{proposition}
\begin{proof}
By \eqref{c14a}  and Proposition \ref{e3b} we get that $u_{yy}(0,0)\ne0$.\\
Next we prove \eqref{e18}. Since $u_x\not\equiv0$, as in Proposition \eqref{e3b}, let us consider the integer $n\ge3$ such that
 \eqref{c2a} holds. Then we have that 
\begin{equation}\label{C14}
\frac{\partial^mu}{\partial x^{m-k}\partial y^k}(0,0)=0\quad\hbox{for any $m<n$ and }k=0,..,m-1.
\end{equation}
Note that, unlike what happens in \eqref{C4a} of Proposition \ref{e3b},  since $u_{yy}(0,0)\ne0$ 
\eqref{C14}  does not hold for $k=m$. 
Next let us differentiate  \eqref{P} $n-2$ times with respect to $x$ and $y$. We get that
\begin{equation}\label{c15}
\frac{\partial^nu}{\partial x^{n-k}\partial y^k}(0,0)=-\frac{\partial^nu}{\partial x^{n+2-k}\partial y^{k-2}}(0,0)+\frac{\partial^{n-2}\big(f(u)\big)}{\partial x^{n-k}\partial y^{k-2}}(0,0)
\qquad\hbox{ for any } k=2,..,n.
\end{equation}
Observe that in $\frac{\partial^{n-2}\big(f(u)\big)}{\partial x^{n-k}\partial y^{k-2}}(0,0)$ appear the derivatives of $u$ with order $1,..,n-2$. All these are zero by \eqref{C14} except when $k=n$. So
we have 
\begin{equation}\label{c16}
\frac{\partial^{n-2}\big(f(u)\big)}{\partial x^{n-k}\partial y^{k-2}}(0,0)=0
\qquad\hbox{ for any } k=2,..,n-1.
\end{equation}
As in the proof of  Proposition \ref{e3b} we get from \eqref{c15} and  \eqref{c16},
\begin{equation}\label{c17}
\begin{cases}
\frac{\partial^nu}{\partial x^n}(0,0)=-\frac{\partial^nu}{\partial x^{n-2}y^2}(0,0)=\dots\Rightarrow
\frac{\partial^nu}{\partial x^{n-2k}\partial y^{2k}}(0,0)=(-1)^k\frac{\partial^nu}{\partial x^n}(0,0),\quad k=1,..,\left[\frac{n-1}2\right],
\\ \\
\frac{\partial^nu}{\partial x^{n-1}\partial y}(0,0)=-\frac{\partial^nu}{\partial x^{n-3}\partial y^3}(0,0)=\dots\Rightarrow
\frac{\partial^nu}{\partial x^{n-2k-1}\partial y^{2k+1}}(0,0)=(-1)^k\frac{\partial^nu}{\partial x^{n-1}\partial y}(0,0),\ k=1,..,\left[\frac{n-2}2\right],
\end{cases}
\end{equation}
By the previous equality and the definition of $n$ we get \eqref{C13}.

Now we use the previous identities  in Taylor's formula. We get
\begin{equation}\label{c18}
\begin{split}
&u(x,y)=u(0,0)+\frac{u_{yy}(0,0)}2y^2+\sum\limits_{m=3}^n\frac{\partial^mu}{\partial y^m}(0,0)y^m+
\frac1{n!}\sum\limits_{k=0}^{n-1}\left(\begin{matrix}n\\k\end{matrix}\right)
\frac{\partial^nu}{\partial x^{n-k}\partial y^k}(0,0)x^{n-k}y^k+R(x,y)=\\
&u(0,0)+\frac{u_{yy}(0,0)+o(1)}2y^2+
\frac1{n!}\sum\limits_{k=0}^{\left[\frac n2\right]}\left(\begin{matrix}n\\2k\end{matrix}\right)
\frac{\partial^nu}{\partial x^{n-2k}\partial y^{2k}}(0,0)x^{n-2k}y^{2k}+\\
&\frac1{n!}\sum\limits_{k=0}^{\left[\frac{n-1}2\right]}\left(\begin{matrix}n\\2k+1\end{matrix}\right)
\frac{\partial^nu}{\partial x^{n-2k-1}\partial y^{2k+1}}(0,0)x^{n-2k-1}y^{2k+1}+R(x,y)=\\
&(\hbox{using \eqref{c17}})=\\
&=u(0,0)+\frac{u_{yy}(0,0)+o(1)}2y^2+
\frac{\frac{\partial^nu}{\partial x^n}(0,0)}{n!}\sum\limits_{k=0}^{\left[\frac n2\right]}\left(\begin{matrix}n\\2k\end{matrix}\right)(-1)^ky^{n-2k}x^{2k}+\\
&\frac{\frac{\partial^nu}{\partial x^{n-1}\partial y}(0,0)}{n!}\sum\limits_{k=0}^{\left[\frac{n-1}2\right]}\left(\begin{matrix}n\\2k+1\end{matrix}\right)
(-1)^kx^{n-2k-1}y^{2k+1}+R(x,y)=\\
&=u(0,0)+\frac{u_{yy}(0,0)+o(1)}2y^2+\frac{\frac{\partial^nu}{\partial x^n}(0,0)}{n!}Re\left(z^n\right)+
\frac{\frac{\partial^nu}{\partial x^{n-1}\partial y}(0,0)}{n!}Im\left(z^n\right)+R(x,y).
\end{split}
\end{equation}
with $R(x,y)=O\left(|x|^n+|y|^n\right)$. This ends the proof.\\
\end{proof}
Now we are in position to give the proof of our theorems.
\begin{proof}[\bf Proof of Theorem \ref{c2}]
Since $(0,0)$ is a maximum point by Theorem \ref{b3} we get that   $i[\nabla u,(0,0)]=1$. Hence
Propositions \ref{e3b} implies \eqref{c1a}. Next to verify \eqref{c1b} and \eqref{c1c}  we consider the cases $n$ even and $n$ odd separately.
\vskip0.3cm\noindent
{\bf The case $n$ even}\\
From \eqref{e18} we have to show that $\frac{\partial^nu}{\partial x^n}(0,0)<0$. If
\begin{equation}
\frac{\partial^nu}{\partial x^n}(0,0)\ne0
\end{equation}
taking $y=0$ in \eqref{e18} and using that $(0,0)$ is a maximum point we get
\begin{equation}
u(0,0)\ge u(x,0)=u(0,0)+\frac{\frac{\partial^nu}{\partial x^n}(0,0)}{n!}x^n+O(|x|^{n+1}),
\end{equation}
and because $n$ is even the claim follows by choosing $|x|$ small enough. Hence suppose that $\frac{\partial^nu}{\partial x^n}(0,0)=0$. By \eqref{C13} we have that $\frac{\partial^nu}{\partial x^{n-1}\partial y}(0,0)\ne0$ and choosing $y=Bx^{n-1}$ in  \eqref{e18} we get
\begin{equation}
\begin{split}
&u(0,0)\ge u(x,Bx^{n-1})=u(0,0)+\left(B^2\frac{u_{yy}(0,0)+o(1)}2+B\frac{\frac{\partial^nu}{\partial x^{n-1}\partial y}(0,0)}{n!}\right)x^{2n-2}+o(|x|^{2n-2})>\\
&u(0,0)
\end{split}
\end{equation}
if  $B$ satisfies $B^2\frac{u_{yy}(0,0)+o(1)}2+B\frac{\frac{\partial^nu}{\partial x^{n-1}\partial y}(0,0)}{n!}>0$. This gives a contradiction.
\vskip0.3cm\noindent
{\bf The case $n$ odd}\\
First let us prove that ${\frac{\partial^nu}{\partial x^n}(0,0)}=0$. As in the previous case let us test \eqref{e18} for $y=0$. We get
\begin{equation}
u(0,0)\le u(x,y)=u(0,0)+\frac{\frac{\partial^nu}{\partial x^n}(0,0)}{n!}x^n+O(|x|^{n+1}),
\end{equation}
and since {\em $n$ is odd} we necessarily have that ${\frac{\partial^nu}{\partial x^n}(0,0)}=0$ and by \eqref{C13}  we deduce  that ${\frac{\partial^nu}{\partial x^{n-1}\partial y}(0,0)}\ne0$. \\
Now we prove \eqref{c1c} and \eqref{c1d}. We claim that
\begin{equation}\label{c21}
\hbox{there exists $l\le2n-2$ such that }\frac{\partial^lu}{\partial x^l}(0,0)\ne0.
\end{equation}
By contradiction let us suppose that 
\begin{equation}\label{e20}
\frac{\partial^{n+1}u}{\partial x^{n+1}}(0,0)=
\frac{\partial^{n+2}u}{\partial x^{n+2}}(0,0)=..=\frac{\partial^{2n-2}u}{\partial x^{2n-2}}(0,0)=0
\end{equation}
and complete  the expansion in \eqref{e18}  adding the terms up to the order $2n-2$. We get, 
\begin{equation}\label{c19}
\begin{split}
&u(x,y)=u(0,0)+\frac{u_{yy}(0,0)+o(1)}2y^2+
\frac{\frac{\partial^nu}{\partial x^{n-1}\partial y}(0,0)}{n!}Im\left(z^n\right)+\\
&\frac1{(n+1)!}\sum\limits_{k=1}^{n+1}\left(\begin{matrix}n+1\\k\end{matrix}\right)
\frac{\partial^{n+1}u}{\partial x^{n+1-k}\partial y^k}(0,0)x^{n+1-k}y^k+\dots+\\
&\frac1{(2n)!}\sum\limits_{k=1}^{2n}\left(\begin{matrix}2n\\k\end{matrix}\right)
\frac{\partial^{2n}u}{\partial x^{2n-k}\partial y^k}(0,0)x^{2n-k}y^k+R(x,y),
\end{split}
\end{equation}
with $R(x,y)=O\left(|x|^{2n+1}+|y|^{2n+1}\right)$. Setting $y=Ax^{n-1}$ in  \eqref{c19} ($A>0$ will be chosen later) and observing that $l-k+(n-1)k>2n-2$ for any $k\ge1$ and  $l\in[n+1,2n-2]$ we derive that
\begin{equation}
\begin{split}
&u(x,x^{n-1})=u(0,0)+\left(\frac{u_{yy}(0,0)}2A^2+
\frac{\frac{\partial^nu}{\partial x^{n-1}\partial y}(0,0)}{(n-1)!}A+o(1)\right)x^{2n-2}+o\left(x^{2n-2}\right)
\end{split}
\end{equation}
Choosing $\frac{u_{yy}(0,0)}2A^2+
\frac{\frac{\partial^nu}{\partial x^{n-1}\partial y}(0,0)}{(n-1)!}A>0$ and $|x|$ small enough we have that $u(x,x^{n-1})>u(0,0)$ and it contradicts that $(0,0)$ is a maximum point. Hence \eqref{c21} holds and choosing the minimum integer $l$ in \eqref{c21} we get that \eqref{c19} becomes
\begin{equation}
\begin{split}
&u(x,y)=u(0,0)+\frac{u_{yy}(0,0)+o(1)}2y^2+
\left(\frac{\frac{\partial^nu}{\partial x^{n-1}y}(0,0)}{n!}+o(1)\right)Im\left(z^n\right)+\\
&\left(\frac1{l!}+o(1)\right)\frac{\partial^lu}{\partial x^l}(0,0)x^l.
\end{split}
\end{equation}
which is equivalent to \eqref{c1c}. Since $(0,0)$ is a maximum point we get that $\frac{\partial^lu}{\partial x^l}(0,0)<0$ for $l\in[n+1,2n-2]$ and finally if $l=2n-2$, setting $y=Ax^{n-1}$ in  \eqref{c1c}, we get
\begin{equation}
\begin{split}
&u(x,y)=u(0,0)+\left(\frac{u_{yy}(0,0)}2A^2+
\frac{\frac{\partial^nu}{\partial x^{n-1}y}(0,0)}{(n-1)!}A+\frac{\frac{\partial^lu}{\partial x^l}(0,0)}{l!}+o(1)
\right)x^{2n-2}
\end{split}
\end{equation}
which implies
\begin{equation}
\frac{u_{yy}(0,0)}2A^2+
\frac{\frac{\partial^nu}{\partial x^{n-1}y}(0,0)}{(n-1)!}A+\frac{\frac{\partial^lu}{\partial x^l}(0,0)}{l!}\le0\quad\hbox{for any }A\in\R.
\end{equation}
Then \eqref{c1d}  holds and this ends the proof.
\end{proof}
\vskip0.2cm
\begin{proof}[\bf Proof of Theorem \ref{c3b}]
Since $i[\nabla f,(0,0)]\ge-1$ Proposition \ref{e3b} applies and so $u_{yy}(0,0)\ne0$. \\
In order to prove \eqref{c4c}-\eqref{e1c} we first remark that $u_x\not\equiv0$. Otherwise we get that $u=u(y)$ verifies $-u''=f(u)$ and since $u_{yy}\ne0$ we deduce that $(0,0)$ is a maximum or minimum point for $u$. So Proposition \eqref{c17b} applies and as in the proof of Theorem \ref{c2}  we use \eqref{e18} to handle the cases $n$ even and $n$ odd. However here the condition
\begin{equation}\label{C20}
i(\nabla u,(0,0))=-1
\end{equation}
plays a crucial role.\\
{\bf The case $n$ even}\\
From \eqref{e18}  we get that
\begin{equation}\label{c20}
\begin{cases}
&u_x=\frac{\frac{\partial^nu}{\partial x^n}(0,0)}{(n-1)!}Re\left(z^{n-1}\right)+\frac{\frac{\partial^nu}{\partial x^{n-1}\partial y}(0,0)}{(n-1)!}Im\left(z^{n-1}\right)+\frac{\partial R(x,y)}{\partial x}
\\
&u_y=(u_{yy}(0,0)+o(1))y-\frac{\frac{\partial^nu}{\partial x^n}(0,0)}{(n-1)!}Im\left(z^{n-1}\right)+
n\frac{\frac{\partial^nu}{\partial x^{n-1}\partial y}(0,0)}{(n-1)!}Re\left(z^{n-1}\right)+\frac{\partial R(x,y)}{\partial y},
\end{cases}
\end{equation}
Then  let us consider the vector field $f(x,y)=(f_1(x,y),f_2(x,y))$  given by
\begin{equation}\label{e21}
\begin{cases}
&f_1(x,y)=\frac{\frac{\partial^nu}{\partial x^n}(0,0)}{(n-1)!}Re\left(z^{n-1}\right)+\frac{\frac{\partial^nu}{\partial x^{n-1}\partial y}(0,0)}{(n-1)!}Im\left(z^{n-1}\right)+\frac{\partial R(x,y)}{\partial x}
\\
&f_2(x,y)=u_{yy}(0,0)y-\frac{\frac{\partial^nu}{\partial x^n}(0,0)}{(n-1)!}Im\left(z^{n-1}\right)+
\frac{\frac{\partial^nu}{\partial x^{n-1}\partial y}(0,0)}{(n-1)!}Re\left(z^{n-1}\right)+\frac{\partial R(x,y)}{\partial y},
\end{cases}
\end{equation}
If $\frac{\partial^nu}{\partial x^n}(0,0)\ne0$ then  $f$ satisfied the assumptions of  Lemma \ref{c12} with\\
 $m=n-1$,\\
 $L(x,y)=(0,u_{yy}(0)y)$,\\
$H_{n-1}(x,y)=\left(\frac{\frac{\partial^nu}{\partial x^n}(0,0)}{(n-1)!}Re\left(z^{n-1}\right)+\frac{\frac{\partial^nu}{\partial x^{n-1}\partial y}(0,0)}{(n-1)!}Im\left(z^{n-1}\right),-\frac{\frac{\partial^nu}{\partial x^n}(0,0)}{(n-1)!}Im\left(z^{n-1}\right)+\frac{\frac{\partial^nu}{\partial x^{n-1}\partial y}(0,0)}{(n-1)!}Re\left(z^{n-1}\right)\right)$,\\ 
$R(x,y)=\left(\frac{\partial R(x,y)}{\partial x},\frac{\partial R(x,y)}{\partial y}\right)$.\\
Moreover\\
$N(L)=(x,0)$ with $x\in\R$ and $R(L)=(0,y)$ with $y\in\R$,\\
$Q(x,y)=(x,0)=P(x,y)$,\\
$J=I$,\\
We claim that the main assumption in Lemma \ref{c12} , $H_{n-1}\Big|_{N(L)\setminus0}\not\in R(L)$  is satisfied. Indeed if $(x,y)\in N(L)\setminus0$ then $y=0$ and $x\ne0$. Hence
\begin{equation}
H_{n-1}(x,0)=\frac1{(n-1)!}\left(\frac{\partial^nu}{\partial x^n}(0,0)x^{n-1},\frac{\partial^nu}{\partial x^{n-1}y}(0,0)x^{n-1}\right)
\end{equation}
and since $\frac{\partial^nu}{\partial x^n}(0,0)\ne0$ this does not belong to $R(L)=(0,y)$. Then the claim of  Lemma \ref{c12} gets that
\begin{equation}
i[f,(0,0)]=i[L+JP,0]i\left[J^{-1}QH_{n-1}\Big|_{N(L)},0\right],
\end{equation}
and since in our case $(L+JP)(x,y)=L(x,y)+P(x,y)=(0,u_{yy}(0,0)y)+(x,0)=\big(x,u_{yy}(0,0)y\big)$ we get that 
\begin{equation}
i[L+JP,(0,0)]=\hbox{sgn}\big(u_{yy}(0,0)\big)
\end{equation}
On the other hand we have that $\big(J^{-1}QH_{n-1}\big)(x,0)=Q\frac1{(n-1)!}\left(\frac{\partial^nu}{\partial x^n}(0,0)x^{n-1},\frac{\partial^nu}{\partial x^{n-1}y}(0,0)x^{n-1}\right)=\left(\frac1{(n-1)!}\frac{\partial^nu}{\partial x^n}(0,0)x^{n-1},0\right)$ and since $n$ is even
\begin{equation}
i\left[J^{-1}QH_{n-1}\Big|_{N(L)},(0,0)\right]=\hbox{sgn}\left(\frac{\partial^nu}{\partial x^n}(0,0)\right)
\end{equation}
Hence we get that 
\begin{equation}
i[f,(0,0)]=\hbox{sgn}\big(u_{yy}(0,0)\big)\hbox{sgn}\left(\frac{\partial^nu}{\partial x^n}(0,0)\right).
\end{equation}
Finally by \eqref{c20} and \eqref{e21} we get that $\nabla u$ is a small perturbation of the vector field $f$ in a suitable small ball centered at $(0,0)$. Then  $i[f,(0,0)]=i[\nabla u,(0,0)]$ and since by assumption $i[\nabla u,(0,0)]=-1$ then \eqref{c4a} follows.\\
{\bf The case $n$ odd}\\
Here we proceed as in the case where $n$ is even. Let us suppose that $\frac{\partial^nu}{\partial x^n}(0,0)\ne0$ and
using the same notation we get that $i[f,(0,0)]$ is well posed and again
\begin{equation}
i[f,(0,0)]=i[L+JP,0]i\left[J^{-1}QH_{n-1}\Big|_{N(L)},0\right].
\end{equation}
However since in this case $n$ is odd we have that $\big(J^{-1}QH_{n-1}\big)(x,0)=0$ and so $i[f,(0,0)]=0$. As in the previous case this implies $i[\nabla u,(0,0)]=0$, a contradiction. Then we have that
\begin{equation}\label{c22}
\frac{\partial^nu}{\partial x^n}(0,0)=0.
\end{equation}
Hence \eqref{e18} and \eqref{c22} imply \eqref{e1c}.
\end{proof}
\sezione{Nondegeneracy of solutions  with one critical point}\label{s4}
In this section we consider a solution $u$ to \eqref{P+} with $f(0,0)\ge0$ and $\partial\Om$ with positive curvature. Moreover we assume that $(0,0)\in\Om$ is 
the {\bf unique} critical point of $u$ (of course its maximum).

First let us recall some results proved in \cite{CC}. For $\th\in[0,\pi)$ set
\begin{equation}
u_\th=\cos\th u_x+\sin\th u_y,
\end{equation}
and
\begin{equation}
\mathcal{N_\th}=\left\{(x,y)\in\Om\hbox{ such that }u_\th(x,y)=0\right\}.
\end{equation}
We have the following result,
\begin{lemma}\label{d3a}
We have that  $\mathcal{N_\th}\cap\partial\Om$ consists of exactly two points $P_1$ and $P_2$ and, around each
$P_i$, $\mathcal{N_\th}$ is a smooth curve that intersects $\partial\Om$ transversally at its end-point $P_i$, $i=1,2$.
\end{lemma}
\begin{proof}
Let $P\in\mathcal{N_\th}\cap\partial\Om$ and $\nu=(\nu_x,\nu_y)$ the exterior unit normal to $\Om$. Then, since $u=0$ on $\partial\Om$ we have that $u_\th(P)=\frac{\partial u}{\partial \nu}(P)(cos\th\nu_x(P)+\sin\th\nu_y(P))$. Since $f(0)\ge0$ 
by Hopf Lemma we get that $\frac{\partial u}{\partial \nu}(P)<0$ and so $u_\th(P)=0$ if and only if the vector $(cos\th,\sin\th)$ is orthogonal to the normal $\nu$ and since the curvature of $\partial\Om$ is positive  it happens exactly at two points $P_1$ and $P_2$.\\
Next let us show that around  each $P_i$, $i=1,2$, $\mathcal{N_\th}$ is a smooth curve. We have that 
\begin{equation}\label{a2}
<\nabla u_\th,(cos\th,\sin\th)>=\cos^2\th u_{xx}+2\cos\th\sin\th u_{xy}+\sin^2\th u_{yy}
\end{equation}
and using that $u=0$ on $\partial\Om$ and $(cos\th,\sin\th)$ is orthogonal to $\nu$ we get  that
\begin{equation}
\cos\th=\frac{u_y(P)}{\frac{\partial u}{\partial \nu}(P)}\quad\hbox{and }\sin\th=-\frac{u_x(P)}{\frac{\partial u}{\partial \nu}(P)}.
\end{equation}
Hence \eqref{a2} becomes
\begin{equation}
<\nabla u_\th,(cos\th,\sin\th)>=u_y^2(P)u_{xx}+2u_xu_y(P) u_{xy}+u_x^2(P)u_{yy}\ne0
\end{equation}
because the curvature of $\partial\Om$ is positive. Then the implicit function theorem applies and the claim follows.
\end{proof}
\begin{remark}\label{d7}
The condition $f(0)\ge0$ is used to apply the Hopf Lemma on $\partial\Om$ and deduce that $\frac{\partial u}{\partial\nu}<0$. What we need in the proof of Lemma \ref{d3a} is that $\nabla u\ne(0,0)$ on $\partial\Om$.
\end{remark}
We say that $P\in\Om$ is a singular point for $u_\th$ if
\begin{equation}
\nabla u_\th(P)=(0,0).
\end{equation}
\begin{lemma}\label{L}
If $P\in\mathcal{N_\th}$ is a singular point for $u_\th$ then the curve  $\mathcal{N_\th}$ ``encloses'' a subdomain $\omega\subset\Om$. More precisely 
there exists a subdomain $\omega\subset\Om$ such that
\begin{itemize}
\item $\partial\omega$ is a Jordan curve.
\item $\partial\omega\cap\partial\Omega=\emptyset$
\item $\partial\omega\subset\mathcal{N_\th}$.
\end{itemize}
\end{lemma}
\begin{proof}
See \cite{CC}.
\end{proof}
Now we are in position to prove Theorem \ref{T1}.
\begin{proof}[\bf Proof of Theorem \ref{T1}]
By contradiction let us suppose that $(0,0)$ is degenerate. Up to a rotation we can assume \eqref{c1} and then
\begin{equation}
\nabla u_x(0,0)=0.
\end{equation}
By Lemma \ref{L} the curve $\mathcal{N}_0$ encloses a region $\omega\subset\subset\Om$.
Moreover  $u_x$ verifies
\begin{equation}\label{4}
\begin{cases}
-\Delta u_x=f'(u)u_x & \hbox{ in }\omega\\
u_x=0 & \hbox{ on }\de\omega
\end{cases}
\end{equation}
and we can suppose  that $u_x>0$  in $\omega$.
Next let us take $Q\in\omega$ and  set 
\begin{equation}
w(x,y)=u_y(Q)u_x(x,y)-u_x(Q)u_y(x,y)
\end{equation}
and
$$\mathcal{M}=\left\{(x,y)\in\Om\hbox{ such that }w(x,y)=0\right\}.$$
Of course $Q\in\mathcal{M}$ and since $\nabla u(Q)\ne(0,0)$ (uniqueness of the critical point) the definition of $w$ is well posed. Moreover $w$ satisfies
\begin{equation}
-\Delta w=f'(u)w \quad \hbox{ in }\Om
\end{equation}
and by \cite{CF} we have that around $Q$ the set $\mathcal{M}$ consists of a finite number of  curves intersecting transversally at $Q$ (this is actually a smooth curve is $\nabla w(Q)\ne0$).\\
We have the following alternative,\\
{\bf Case $1$: $\mathcal{M\cap\de\omega}$ is given by at most one point}\\
In this case we have that $\mathcal{M}$ encloses at least one region in $\omega$ and then there exists $D\subset\omega$ such that 
\begin{equation}\label{d4}
\begin{cases}
-\Delta w=f'(u)w & \hbox{ in }D\\
w=0 & \hbox{ on }\de D
\end{cases}
\end{equation}
This implies that the first eigenvalue $\lambda_1(-\Delta-f'(u)I)<0$ in $\omega$ and this gives a contradiction with \eqref{4}.\\
{\bf Case $2$: $\mathcal{M\cap\de\omega}$ is given by at least two  points}\\
In this case we would have that the derivatives $u_\th$ and $w$ intersect at two different points and this gives a contradiction with the uniqueness of the critical point unless $w=Cu_x$.
This implies that $u_\th(Q)=Cw(Q)=0$ and so $Q\in\partial\omega$. This is a contradiction which ends the proof.
\end{proof}
\begin{proof}[\bf Proof of Corollary \ref{C1}]
Let us suppose that $(0,0)$ is a degenerate critical point of $u$ and consider a level set ${\mathcal C}=\{u=c\}$ with $0<c<||u||_\infty$. Since $(0,0)$ is the unique critical point we have that ${\mathcal C}$ is a smooth closed curve. By contradiction suppose that the curvature of  ${\mathcal C}$ is
positive everywhere. We have that $u$ solves to the following problem,
\begin{equation}
\begin{cases}
-\Delta u=f(u) & \hbox{ in }\{u>c\}\\
u=c & \hbox{ on }{\mathcal C}.
\end{cases}
\end{equation}
The same proof  Theorem \ref{T1} applies with $\partial\Omega$ replaced by ${\mathcal C}$. The only difference is that here the condition $f(0)\ge0$ is not
sufficient to apply  the Hopf Lemma to ${\mathcal C}$ in Lemma \ref{d3a} because $c\ne0$. However if  $\frac{\partial u}{\partial \nu}(P)=0$ for some $P\in{\mathcal C}$ then $\nabla u(P)=0$ and this contradicts the uniqueness of the critical point. The claim of Theorem \ref{T1} gives that $(0,0)$ is nondegenerate,
a contradiction. This means that $any$ level set $\{u=c\}$ must contain a point with nonpositive curvature. This ends the proof.
\end{proof}
\begin{proof}[\bf Proof of Corollary \ref{C2}]
It is the same of Corollary \ref{C1}.
\end{proof}
We end this section with two interesting examples of solution to \eqref{P+}, both suggested by
 Theorems \ref{c2} and \ref{c3b}.  The first one shows that \eqref{d5} can occur.
\begin{example}\label{e5}
If we put $n=4$ in \eqref{c1b} we are leading to consider the set
\begin{equation}
D_\al=\left\{(x,y)\in\R^2\hbox{ such that }
\frac12y^2+x^4-6x^2y^2+y^4<\al\right\}.
\end{equation}
It  is not difficult to show that there exists $\al_0>0$ such that  for any $0<\al<\al_0$ the set $\partial D_\al$ is a closed curve. Note that $\partial D_\al$ shrinks to the origin as $\al\to0$.
Fix $\bar\al\in(0,\al_0)$ and let us consider the function
\begin{equation}
u(x,y)=\bar\al-
\left(\frac12y^2+x^4-6x^2y^2+y^4\right).
\end{equation}
We have that $u$ verifies
\begin{equation}\label{C30}
\begin{cases}
-\Delta u=1 & \hbox{ in }D_{\bar\al}\\
u>0& \hbox{ in }D_{\bar\al}\\
u=0& \hbox{ on }\partial D_{\bar\al}
\end{cases}
\end{equation}
and $(0,0)$ is the unique critical (maximum) point to $u$ in $D_{\bar\al}$ provided to choose $\bar\al$ suitably smaller. Moreover $u$ is a semi-stable solution to \eqref{C30}.
A straightforward computation shows that the curvature at a point $(0,y)\in\{u=c\}$ with $0<c<\bar\al$  is given by
\begin{equation}
\begin{split}
&k(0,y)=-\frac{u_{xx}u_y^2-2u_{xy}u_xu_y+u_{yy}u_x^2}{(u_x^2+u_y^2)^\frac32}=-\frac{12|y|}{1+4y^3}<0
\end{split}
\end{equation}
for $y<0$. This means that there is always at least a point in $\{u=c\}$ with negative curvature which gives the claim.
\end{example}

Next example shows that semi-stable solutions to \eqref{P+} can have two critical points in star-shaped domains. 
\begin{example}\label{d6}
Theorem \ref{c3b} with $n=3$ suggested the construction of this example.\\
For $c$ small enough let us consider
\begin{equation}
\begin{split}
&D_c=\left\{(x,y)\in\R^2\hbox{ such that }\frac12y^2+x^3-3xy^2+A(x^4-6x^2y^2+y^4)<c\right\}
\end{split}
\end{equation}
Some tedious computations show that for $A>\frac{35}2$ and $c$ small enough $D_c$ is a closed curve.\\
For the same parameters $A$ and $c$ let us consider the function
\begin{equation}
u(x,y)=c-
\left(\frac12y^2+x^3-3xy^2+A(x^4-6x^2y^2+y^4)\right).
\end{equation}
We have that $u$ verifies
\begin{equation}
\begin{cases}
-\Delta u=1 & \hbox{ in }D_c\\
u>0& \hbox{ in }D_c\\
u=0& \hbox{ on }\partial D_c
\end{cases}
\end{equation}
A straightforward computation shows that $u$ admits exactly two critical points. $P_0=(0,0)$ (saddle point) and $P_1=\left(-\frac3{4A},0\right)$ (maximum point). As remarked in the introduction, $D_c$ is a star-shaped domain with respect to the origin. Since $u$ is a semi-stable solution to \eqref{P+}, this example shows that in Cabr\'e-Chanillo's Theorem the assumption of positive curvature cannot be relaxed to star-shapedness with respect to some point.
\end{example}
\sezione{Non-isolated maximum points}\label{s5}
In  this section we show some consequences  of Theorem \ref{c2} when the maximum (minimum) points of $u$ is not isolated. 
\begin{proof}[\bf Proof of Proposition \ref{P1}]
Since $(0,0)$ is not isolated we have that the sets $S_\e=\left\{x^2+y^2=\e\right\}$ and $\mathcal{C}=\{(x,y)\in B_1\hbox{ such that }u(x,y)=u(0,0)\}$ verify
 \begin{equation}\label{d13}
 S_\e\cap\mathcal{C}\ne\emptyset
\end{equation}
for $\e>0$ small enough. Let us consider $(x_\e,y_e)\in S_\e\cap\mathcal{C}$.  If by contradiction $n$ is even in Theorem \ref{c2}, using that  for $(x,y)\in S_\e$,
 \begin{equation}
Re(z^n)=x^n+o(y^2)\quad\hbox{and}\quad Im(z^n)=nx^{n-1}y+o(y^2),
\end{equation}
we get by \eqref{c1b} and  \eqref{d13}
 \begin{equation}
  \begin{split}
&0=\left(\frac{u_{yy}(0,0)}2+o(1)\right)y_\e^2+\left(\frac{\frac{\partial^nu}{\partial x^n}(0,0)}{n!}+o(1)\right)\left(\e^2-y_\e^2\right)^\frac n2+\left(\frac{\frac{\partial^nu}{\partial x^{n-1}\partial y}(0,0)}{(n-1)!}+o(1)\right)x_\e^{n-1}y_\e=\\
&\big(\hbox{using the Young inequality}\big)=\\
&\left(\frac{u_{yy}(0,0)}2+o(1)\right)y_\e^2+\left(\frac{\frac{\partial^nu}{\partial x^n}(0,0)}{n!}+o(1)\right)\e^n,
\end{split}
\end{equation}
which implies $y_\e=\e=0$, a contradiction with \eqref{d13}. 

So $n$ is odd and repeating the previous computation for   \eqref{c1c}, we obtain
\begin{equation}\label{d14}
  \begin{split}
&0=\left(\frac{u_{yy}(0,0)}2+o(1)\right)y_\e^2+\left(
\frac{\frac{\partial^nu}{\partial x^{n-1}y}(0,0)}{(n-1)!}+o(1)\right)y_\e\left(\e^2-y_\e^2\right)^\frac{n-1}2+\left(\frac{\frac{\partial^lu}{\partial x^l}(0,0)}{l!}+o(1)\right)\left(\e^2-y_\e^2\right)^\frac l2=\\
&\left(\frac{u_{yy}(0,0)}2+o(1)\right)y_\e^2+\left(
\frac{\frac{\partial^nu}{\partial x^{n-1}y}(0,0)}{(n-1)!}+o(1)\right)\e^{n-1}y_\e+\left(\frac{\frac{\partial^lu}{\partial x^l}(0,0)}{l!}+o(1)\right)\e^l.
\end{split}
\end{equation}
From \eqref{d14} we deduce that
\begin{equation}\label{d15}
\frac2{l!}u_{yy}(0,0){\frac{\partial^lu}{\partial x^l}(0,0)}\e^l-\frac1{(n-1)!^2}\left(
\frac{\partial^nu}{\partial x^{n-1}y}(0,0)\right)^2\e^{2n-2}\le0
\end{equation}
and if $l<2n-2$ we get $u_{yy}(0,0){\frac{\partial^lu}{\partial x^l}(0,0)}\le0$, a contradiction with Theorem \ref{c2}. So  $l=2n-2$ in \eqref{d15} and using \eqref{c1d}  we deduce that
\begin{equation}
\frac2{l!}u_{yy}(0,0){\frac{\partial^lu}{\partial x^l}(0,0)}=\frac1{(n-1)!^2}\left(
\frac{\partial^nu}{\partial x^{n-1}y}(0,0)\right)^2
\end{equation}
which ends the proof.
\end{proof}
\begin{proof}[\bf Proof of Corollary \ref{C3}]
Assume that $u=u(r)$ is a radial function  with $r^2=x^2+(y-P)^2$, $P\ne0$ and $u$ solves 
\begin{equation}
-\Delta u=f(u)
\end{equation}
Suppose that $u'(P)=0$ and $u''(P)<0$ and so $(0,0)$ belongs to a circle of maximum points.
The derivatives at $(0,0)$ are given by 
\begin{equation}\label{d16}
\begin{cases}
u_{yy}(0,0)=u''(P)=-f(u(P))\\
u_{xxx}(0,0)=0\\ 
u_{xxy}(0,0)=-\frac1Pu''(P)=\frac1Pf(u(P))\\ 
u_{xxxx}(0,0)=\frac3{P^2}u''(P)=-\frac3{P^2}f(u(P))
\end{cases}
\end{equation}
So we have that by \eqref{c1c} we get \eqref{e6} and by \eqref{d16} we have
\begin{equation}
u_{xxy}^2(0,0)=\frac{f^2(u(P))}{P^2}=
\frac{u_{yy}(0,0)u_{xxxx}(0,0)}3.
\end{equation}
So the equality is achieved in \eqref{c1d} (equivalently in  \eqref{e7}).
\end{proof}
\begin{remark}
We do know examples of solutions with a curve of maximum point for $l>3$ in Proposition \ref{P1}. It should be interesting to construct it (it there exist!).
\end{remark}
We end this section with some additional properties of solutions with curves of maximum points.
\begin{proposition}\label{D13}
 Assume the same assumptions  of Theorem \ref{c2} and suppose that $\gamma$ is a smooth curve of maxima for $u$. Then, denoting by $k(0,0)$ the curvature of $\gamma$ at $(0,0)$ the following alternative holds: either
\begin{equation}\label{d11}
k(0,0)\ne0
\end{equation}
or
\begin{equation}\label{d12}
u_x\hbox{ consists of $4$ curves intersecting transversally at  $(0,0)$}.
\end{equation}
\end{proposition}
\begin{proof}
Let us suppose that $k(0,0)=0$ and show that \eqref{d12} holds. Let us parametrize $\gamma$ as follows,
\begin{equation}
\gamma(t)=
\begin{cases}
x=x(t)\\
y=y(t)
\end{cases}t\in(-\e,\e)
\end{equation}
with $\gamma(0)=(0,0)$. Then since $u(\gamma(t))=u(0,0)$ for any $t\in(-\e,\e)$ a straightforward  computation  gives
\begin{equation}\label{d9}
\begin{cases}
0=\frac{d^2}{dt^2}u(\gamma(t))\Big|_{t=0}=u_{yy}(0,0)\left(\dot{y}(0)\right)^2
\Rightarrow\dot{y}(0)=0\\ \\
0=\frac{d^4}{dt^4}u(\gamma(t))\Big|_{t=0}=
u_{xxxx}(0,0)\big(\dot{x}(0)\big)^4+5u_{xxy}(0,0)\big(\dot{x}(0)\big)^2\ddot{y}(0)+
3u_{yy}(0,0)\big(\ddot{y}(0)\big)^2
\end{cases}
\end{equation}
Since $k(0,0)=0$, from $\dot{y}(0)=0$  we get that $\ddot{y}(0)=0$ and by \eqref{d9} we deduce that 
\begin{equation}\label{d10}
u_{xxxx}(0,0)=0.
\end{equation}
Since by  Proposition \ref{P1} we have that $n$ is odd,  \eqref{d10} and \eqref{c1d} imply that $n\ge5$
in Theorem \ref{c2}. Finally from \eqref{c1c} we deduce that in a neighborhood of $(0,0)$ it holds
\begin{equation}
u_{x}(x,y)=(A+o(1))Re\left(z^{n-1}\right)+(B+o(1))Im\left(z^{n-1}\right),
\end{equation}
with $n-1\ge4$ and $(A,B)\ne(0,0)$. This gives \eqref{d12} and the claim follows.
\end{proof}
Let us recall that the Morse index $m(u)$ of a solution $u$ to 
\begin{equation}\label{1}
\begin{cases}
-\Delta u=f(u) & \hbox{ in }\Omega\\
u=0 & \hbox{ on }\de \Omega
\end{cases}
\end{equation}
is given by the number of negative eigenvalue  of the operator
\begin{equation}
\mathcal{L}=-\Delta-f'(u)I.
\end{equation}
Our final result is a consequence of Proposition \ref{D13}.
\begin{corollary}
Assume the same assumptions as in Theorem \ref{c2} and suppose that $\gamma$ is a smooth curve of maxima for $u$. Then if $u$ is a solution to \eqref{1} with $m(u)<3$ then 
\begin{equation}
k(0,0)\ne0
\end{equation}
\end{corollary}
\begin{proof}
By Proposition \ref{D13} it is enough to prove that \eqref{d12} cannot occur. Indeed if it happens we deduce that 
\begin{equation}
\mathcal{N}=\{(x,y)\in\Om\hbox{ such that }u_x(x,y)=0\}
\end{equation}
encloses at least three nodal regions $\omega_1, \omega_2, \omega_3\subset\Om$ (straightforward extension of Lemma \ref{L}). So the functions
\begin{equation}
v_i(x,y)=\begin{cases}
u_x& \hbox{ if }(x,y)\in\omega_i\\
0 &  \hbox{ if }(x,y)\in\Om\setminus\omega_i
\end{cases}
i=1,2,3,
\end{equation}
verifies that $v_i\in H^1_0(\Om)$, $\int_\Om\nabla v_i\nabla v_j=0$ for $i\ne j$ and $\int_\Om|\nabla v_i|^2-f'(u)v_i^2=0$ for $i=1,2,3$. Then by the definition of Morse index we derive that $m(u)\ge3$ and this gives a contradiction.
\end{proof}

\end{document}